\documentclass{article}
\usepackage{amssymb}
\usepackage[margin=2.5cm]{geometry}
\usepackage{amsmath}
\usepackage{amsthm}
\usepackage{mathrsfs}
\usepackage{verbatim}
\usepackage{enumerate}
\usepackage{subcaption}
\usepackage{tikz}
\usepackage{listings}
\usepackage{float}
\usepackage{marvosym}
\usepackage{fancyhdr}
\newtheorem{defn}{Definition}[section]
\newtheorem{lemma}[defn]{Lemma}

\newtheorem{cor}[defn]{Corollary}

\newtheorem{prop}[defn]{Proposition}

\newtheorem{theorem}[defn]{Theorem}
\theoremstyle{definition}

\newtheorem{question}[defn]{Question}

\newcommand{\E}{\mathbf{E}}
\newcommand{\Prob}{{\mathbf{P}}}
\newcommand{\rtt}{\mathbf{0}}
\newcommand{\TFM}{{\sf TFM}}

\newcommand{\Poiss}{\mathrm{Poiss}}

\newcommand{\Par}{\overleftarrow}

\begin{document}
	
	\title{The frog model on non-amenable trees}%
	\author{Marcus Michelen\thanks{ Dept. of Mathematics, Statistics and Computer Science, University of Illinois at Chicago, 851 S. Morgan Street, Chicago, IL 60607.} \\
		\texttt{michelen.math@gmail.com}
		\and 
		Josh Rosenberg\thanks{ Dept. of Mathematics,
			University of Washington, C-138 Padelford Hall, Seattle, WA 98195. This author was supported by a Zuckerman STEM Postdoctoral Fellowship, as well as by ISF grant 1207/15, and ERC
			starting grant 676970 RANDGEOM}\\
		\texttt{jr288@uw.edu}
	}
	\date{}

	\maketitle
	
	\begin{abstract}
		We examine an interacting particle system on trees commonly referred to as the frog model.  For its initial state, it begins with a single active particle at the root and i.i.d.~$\mathrm{Poiss}(\lambda)$ many inactive particles at each non-root vertex.  Active particles perform discrete time simple random walk and in the process activate any inactive particles they encounter.  We show that for \emph{every} non-amenable tree with bounded degree there exists a phase transition from transience to recurrence (with a non-trivial intermediate phase sometimes sandwiched in between) as $\lambda$ varies.  
	\end{abstract}
	
	\section{Introduction}
	
	The frog model is a particular system of interacting random walks on a rooted graph.  It starts with a single active particle at the root, and some collection of inactive particles distributed 
	among the non-root vertices.  Active particles perform mutually independent discrete-time simple random walk, and any time an active particle meets a group of inactive particles, the inactive particles become active.  In this system the particles are often referred to as ``frogs,'' where active particles are considered ``awake'' and inactive particles ``sleeping.''  For infinite graphs, studies of the frog model often involve establishing whether it is recurrent (meaning almost surely infinitely many active particles hit the root) or transient (meaning almost surely only finitely many active particles ever hit the root).  Much work has been done on the frog model, including work on $\mathbb{Z}^d$ \cite{alves,DP,gantert-schmidt,GNR,popov-random,TW}; one representative result \cite{alves} shows that the model is recurrent when there are i.i.d.~frogs per vertex.  
	
	Results on trees paint a different picture: Hoffman, Johnson, and Junge in two works \cite{HJJ1,HJJ2} studied two different frog models on regular trees---one considers one frog per vertex while the other considers i.i.d.~$\Poiss(\lambda)$ frogs per vertex---and concluded that ``that the frog model on trees is teetering on the edge between recurrence and transience.''  In particular, they showed that if $\Poiss(\lambda)$ frogs are placed on each vertex of a regular tree, then there is a sharp transition from transience to recurrence as $\lambda$ varies \cite{HJJ1}.  For a more detailed background on the frog model, see \cite{mr} and the references therein. 
	
	A running theme in the previous frog model work is that the underlying graph is transitive or quasi-transitive, meaning that the set of vertices can be partitioned into finitely many sets so that for each pair of vertices in the same set there exists a graph automorphism mapping one to the other.  In a recent work \cite{mr}, the authors studied the frog model with $\Poiss(\lambda)$ frogs per vertex on Galton-Watson trees; under mild assumptions on the offspring distribution, we proved that there is a sharp transition Galton-Watson almost-surely from transience to recurrence as $\lambda$ varies.  The goal of this work is to continue to break free from the restriction of quasi-transitivity and prove the existence of a phase transition for a wide class of deterministic trees.
	
	\subsection{Results}
	
	In the present work we examine the frog model with i.i.d. $\text{Poiss}(\lambda)$ sleeping frogs positioned at each non-root vertex for all non-amenable trees of bounded degree.  In particular, no self-similarity is imposed.  Our main result involves showing that for each tree $T$ in this class, there are critical thresholds $0<\lambda_1(T)\leq\lambda_2(T)<\infty$ such that the model is transient for $\lambda<\lambda_1(T)$ and recurrent for $\lambda>\lambda_2(T)$.   Recall that an infinite graph $G = (V,E)$ of bounded degree is \emph{non-amenable} if for all finite subsets $S \subset V$ the ratio $|\partial S| / |S|$ is uniformly bounded away from $0$; here, $|\partial S|$ is the number of edges with one vertex in $S$ and one in $V \setminus S$ and $|S|$ is the number of vertices in $S$.
	
	\medskip
	\begin{theorem}\label{theorem:mrecth}
		For any non-amenable tree $T$ of bounded degree, there exist $0 < \lambda_1 \leq \lambda_2 < \infty$, such that the frog model on $T$ with i.i.d.\ $\Poiss(\lambda)$ frogs per non-root vertex is transient for $\lambda < \lambda_1$ and recurrent for $\lambda > \lambda_2$.
	\end{theorem}
	
\medskip
Since regular trees are non-amenable and have bounded degree, this result generalizes the results of \cite{HJJ1} to a much wider class.  Unlike the case of regular trees, however, there indeed exist examples of non-amenable trees of bounded degree for which there is a non-trivial intermediate regime, i.e.\ the parameters $\lambda_1$ and $\lambda_2$ from Theorem \ref{theorem:mrecth} satisfy the strict inequality $\lambda_1 < \lambda_2$; an example is provided in Lemma~\ref{lem:nontrivial-intermediate}.  Further, the assumptions of non-amenability and bounded degree cannot be removed: on the amenable tree $\mathbb{Z}$, random walks are recurrent and thus the Poisson frog model has no transient regime; conversely, Lemma \ref{lem:rtnrr} provides an example of a non-amenable tree of unbounded degree without a recurrent regime.
	
As is the case in \cite{mr}, in order to establish recurrence in the proof of Theorem \ref{theorem:mrecth}, we define a separate, more tractable, model on $T$ that we refer to as the truncated frog model.  The main difference between the ordinary frog model and the truncated frog model is that in the latter, trajectories are replaced with their loop-erased versions.  The number of returns to the root in the ordinary frog model will stochastically dominate the number of visits in the truncated model, and so proving recurrence of the truncated frog model will be sufficient to deduce recurrence of the ordinary frog model.  Our proof of recurrence will rely on a bootstrapping argument: we show that if a certain quantity is large, then it can be shown to be even larger.  Our reduction to a loop-erased version as well as our bootstrapping argument certainly incorporate some key ideas from the arguments \cite{HJJ1} and \cite{HJJ2} used for regular trees.  However, the self-similarity in their case allowed them to operate on an extremely local level, while the absence of self-similarity in our case forced us to construct a bootstrapping argument that operates on a larger scale, and indeed each iteration looks at larger and larger neighborhoods.

For the proof of transience, we couple the Poisson frog model on $T$ with branching random walk.  We then use the fact that for a non-amenable graph of bounded degree, the spectral radius is bounded away from $1$.  Then, using a more general result concerning transience of branching Markov chains, we show that for sufficiently small Poisson mean $\lambda$, the branching random walk model is transient on $T$, which by virtue of stochastic dominance, implies that the original frog model is as well.  In the process of obtaining this result, we also derive lower bounds on $\lambda_1(T)$.

	\subsection{Sketch of Recurrence Proof}
	
The proof of recurrence for Theorem \ref{theorem:mrecth} relies on a delicate bootstrapping argument; here, we enumerate some of the key ideas, with the objective of offering a useful road-map through the proof.
	
	\begin{enumerate}
		\item We primarily focus on an altered version of the frog model which we detail in Section \ref{sec:truncated}; its main distinguishing feature is that frogs now perform loop-erased walk instead of simple random walk.  
		
		\item Since the trees we are interested in have bounded degree and no long pipes, the transition probabilities for loop-erased walk are uniformly bounded away from $0$ and $1$, as shown in \eqref{lbpfsn} and \eqref{bfgtransprle}.
		
		\item The main stroke of the proof is to prove Proposition \ref{pr:iterate}: this says that if \begin{equation}\label{eq:intro-recurrence}
		\Prob(u \text{ activated}\,|\,\Par{u}\text{ activated}) \geq 1 - e^{-\lambda i/2}
		\end{equation}
		for all $u$---where $\Par{u}$ denotes the parent of $u$---then in fact the same inequality holds with $i$ replaced by $i + 1$.  The base case follows from step $2$ of the sketch.
		
		\item Choosing $n$ as a function of $i$ and applying \eqref{eq:intro-recurrence} iteratively shows that a constant fraction of vertices at depth $n$ of the descendant tree of any activated vertex are themselves activated with high probability (Claim 2 in the proof of Proposition \ref{pr:iterate}).
		
		\item By using the fact that hitting probabilities are nearly symmetric on the trees of interest (see Appendix \ref{app:hitting}), this implies that if a vertex is activated, on the order of $n$ particles return to it with high probability (Claim 3 in the proof of Proposition \ref{pr:iterate}).
		
		\item If $\Par{u}$ is activated and $u$ is not, some sibling $u'$ of $u$ is.  By the previous step, on the order of $n$ particles return to $u'$ with high probability.  By step $2$, the probability that any given one of these hits $u$ is bounded below, and thus one hits $u$ with high probability, completing the inductive step.
	\end{enumerate}

		\subsection{Description of Non-Amenable Trees}
		
		The class of non-amenable trees of bounded degree may be described without isoperimetric language: 

	\begin{lemma} \label{lem:non-amenable}
		For a given rooted tree $T$, define $T_r$ to be the set of vertices in $T$ through which there exists a non-backtracking path from the root.  If $T$ is of bounded degree, then $T$ is non-amenable if and only if $T$ satisfies both of the following conditions \begin{itemize}
			\item There exists an $M$ so that for each $v \in T$, the connected components of $(T \setminus T_r) \cup \{v\}$ each have at most $M$ vertices;
			\item There exists an $M$ so that there does not exist a path $v_0, v_1, \ldots, v_M$ in $T_r$ with each $v_j$ having degree $2$ in $T_r$.
		\end{itemize} 
	\end{lemma}
	\begin{proof}
		
		Necessity of the two conditions is easier to show: if connected components of $(T\setminus T_r) \cup \{v\}$ can be arbitrarily large, then there exists a sequence $S_n$ of subtrees with $|S_n \cap T_r| = 1$ and $|S_n| \to \infty$.  Such a set must have $|\partial S_n | $ bounded above by the maximum degree of $T$, thereby yielding $|\partial S_n| / |S_n| \to 0$.   Similarly, if paths of vertices of degree $2$ in $T_r$ can be arbitrarily large, let $\gamma_n$ be a collection of such paths so that $|\gamma_n| \to \infty$.  Define $S_n$ to be $\gamma_n$ together with vertices in $T \setminus T_r$ reachable from $\gamma_n$.  Then only the first and last vertex of $\gamma_n$ will be connected to elements of $T \setminus S_n$, thus showing again that $|\partial S_n | / |S_n| \to 0$. 

A quantitative proof of sufficiency is carried out in Lemma \ref{lem:tree-iso}.
\end{proof}

	\section{Recurrence}
	\label{sec:rec-non-amenable}
	In this section, we will establish our recurrence result, which consists of the following theorem.
	
	\medskip
	\begin{theorem}\label{theorem:mrr}
		Let $\mathcal{T}^*_f$ represent the set of all rooted trees of bounded degree that satisfy the second condition of Lemma \ref{lem:non-amenable}.  Then for any tree $T\in\mathcal{T}^*_f$, the frog model on $T$ with i.i.d. $\emph{Poiss}(\lambda)$ frogs per non-root vertex is recurrent for all $\lambda$ sufficiently large.
	\end{theorem}
	
\medskip
To prove Theorem \ref{theorem:mrr}, we begin in Section \ref{sec:truncated} by defining the truncated frog model referenced in the introduction, and then constructing the coupling that is used to show that it is dominated (in terms of the number of returns to the root) by the original model.  Then in Section \ref{sec:looperased} we examine some of the properties of loop erased random walk on trees in $\mathcal{T}^*_f$, achieving several bounds relating to transition probabilities that are needed in order to employ the argument that will be used in the proof of Theorem \ref{theorem:mrr}.  Finally in Section \ref{sec:proof} we present the proof of the theorem, which largely consists of an induction argument relating to the average density of activated vertices on each level of the tree.
	
	\noindent
	\subsection{Simplifications and the truncated frog model} \label{sec:truncated}
	We begin by noting that if we have any tree $T\in\mathcal{T}^*_f$ to which we attach some (possibly infinite) collection of bushes in order to obtain the tree $T'$, then the frog model on $T'$ with $\text{Poiss}(\lambda)$ sleeping frogs per non-root vertex dominates the frog model on $T$ (for the same Poisson mean $\lambda$) with respect to the number of returns to the root (we can see this by simply ignoring the excursions that activated particles originating in $T$ make into $T'\setminus T$).  Hence, if the model with Poisson mean $\lambda$ is recurrent on $T$, then it must also be recurrent on $T'$.  Now combining this with the fact that a rooted tree $T$ is in $\mathcal{T}^*_f$ if and only if its backbone $T_r$ lies in $\mathcal{T}^*_f$, we see that in order to prove the existence of a recurrent regime for all $T\in\mathcal{T}^*_f$, it will suffice to do so just for those $T\in\mathcal{T}^*_f$ that do not have any leaves.  
	
	To further simplify the problem, we now start with a tree $T\in\mathcal{T}^*_f$ that does not have any leaves, and we assign $\text{Poiss}(\lambda)$ sleeping frogs to each non-root vertex in $T$ that has at least two children, while leaving all other non-root vertices unoccupied.  Since the interiors of all pipes now contain no sleeping frogs, this version of the frog model will be unchanged, with respect to the distribution of the number of returns to the root, if we replace each pipe with a single edge to which we assign a resistance equal to the length of the original pipe.  Also noting that we can ignore the case where the root of $T$ has only one child (since infinitely many returns to the first descendant of the root with at least two children will automatically imply almost surely infinitely many returns to the root), as well as the case where the root has exactly two children (since the absence of a recurrent regime on such a tree would suggest that adding an additional child to the root to which we attach some arbitrary infinite subtree would still yield a tree without a recurrent regime), we now see that in order to establish Theorem \ref{theorem:mrr}, it will suffice to establish the existence of a recurrent regime for the frog model on any infinite rooted tree $T$ of bounded degree where all vertices have degree at least three, and to which we assign uniformly bounded positive integer resistances to each of the edges of $T$.
	
	Denoting the collection of rooted trees with weighted edges described in the previous paragraph as $\mathcal{T}_w$, we now proceed to define what we referred to as the truncated frog model in the introduction, on the set of all $T\in\mathcal{T}_w$.  The dynamics of this model (as laid out in \cite{mr}) are as follows:
	
	\begin{enumerate}
		\item Like the first model, this model begins with a single active particle at the root, and i.i.d.\ $\text{Poiss}(\lambda)$ sleeping particles at all non-root vertices.
		
		\item A sleeping particle is activated when the vertex at which it resides is landed on by an active particle.  Upon activation, particles perform independent loop-erased random walks, which terminate upon hitting the root.
		
		\item In addition, any time an active particle takes a step away from the root and lands on a vertex which has already been landed on by at least one other active particle, the particle is eliminated.  If more than one particle simultaneously land on a vertex which had not previously been landed on by an active particle, all but one of these particles are eliminated.
	\end{enumerate}

	\begin{lemma}
		The number of visits to the root in the frog model stochastically dominates the number of visits in the truncated frog model.
	\end{lemma}

	 This is proven in full detail in \cite{mr}, but we provide a short sketch here for completion. 
	\begin{proof}
		If we replace each trajectory with its loop-erased version, this only decreases the set of particles that are activated.  Similarly, removing particles can only decrease the set of activated particles, thereby decreasing the set of particles that hit the root.
	\end{proof}
	\noindent
	\subsection{Loop erased random walk}\label{sec:looperased}
	
	In this section we will let $\{X_j\}$ represent loop erased random walk on $T$, while denoting the probability measure on non-backtracking paths associated with loop-erased random walk beginning at a vertex $v$ as  $\Prob_v$, and letting $p(v,v')$, for distinct vertices $v$ and $v'$, represent the probability that random walk beginning at $v$ (and with transition probabilities in accordance with the assigned edge resistances) ever hits $v'$.  
	
\begin{lemma}
			If $T$ is a weighted tree that has minimum degree $\delta \geq 3$, maximum degree $\Delta$, and all edge resistances in $[1,r]$, then for a loop-erased random walk $\{X_n\}$ on $T$, we have \begin{align}
			\frac{1}{2\Delta r} \leq\ \Prob_v(X_1 = v_i) \leq \frac{r}{r + \delta - 2} \label{lbpfsn} \\
			\mathrm{and}\quad \Prob(X_{n+1} = v_{i'} \,| \,X_n = v, X_{n-1} = v_i) \geq \frac{1}{2 \Delta r^2} \label{bfgtransprle}
			\end{align}  
	\end{lemma}
	\begin{proof}
	First note that
	\begin{equation}\label{transprle}
	\Prob_v(X_1=v_i)=\frac{r^{-1}_i\big(1-p(v_i,v)\big)}{\sum_{j=1}^N r^{-1}_j\big(1-p(v_j,v)\big)}\end{equation}
	where $v_1,\dots,v_N$ represent the neighbors of the vertex $v$ and $r_j$ represents the resistance of the edge connecting $v$ to $v_j$.  In addition, for any $n\geq 1$ and distinct integers $i,i'$ between $1$ and $N$, we have that
	\begin{align}
	\Prob(X_{n+1}=v_{i'}|X_n=v,X_{n-1}=v_i)&=\Prob_{v_i}(X_2=v_{i'}|X_1=v) \nonumber \\
	&=\frac{r^{-1}_{i'}\big(1-p(v_{i'},v)\big)}{\big(1-p(v,v_i)\big)\Big(r^{-1}_i+\underset{j\neq i}{\sum}r^{-1}_j\big(1-p(v_j,v)\big)\Big)}\label{gtransprle}
	\end{align}
	where the absence of a subscript in the expression on the left reflects the fact that the equality holds regardless of the location of $X_0$.  We will work to obtain bounds on the expressions in \eqref{transprle} and \eqref{gtransprle} with respect to $\delta$, $\Delta$, and $r$.  To start, for any pair of neighboring vertices $v,v'$ in $T$, we let $T_{v,v'}$ represent the connected component of $v$ that we get by removing all of the neighbors of $v$ except for $v'$.  Then note that the expression for ${\bf P}_v(X_1=v_i)$ given in \eqref{transprle} can be written as 
	$${\bf P}_v(X_1=v_i)=\frac{\mathscr{R}^{-1}\Big(v\underset{T_{v,v_i}}{\longleftrightarrow}\infty\Big)}{\underset{j\leq N}{\sum}\mathscr{R}^{-1}\Big(v\underset{T_{v,v_j}}{\longleftrightarrow}\infty\Big)}$$
	where $\mathscr{R}\Big(v\underset{T'}{\longleftrightarrow}\infty\Big)$ denotes the effective resistance between $v$ and $\infty$ in the tree $T'$.
	
	To achieve a lower bound for $\mathscr{R}\Big(v\underset{T_{v,v_j}}{\longleftrightarrow}\infty\Big)$ (subject to the constraints that all vertices in the original tree $T$ have degree between $\delta$ and $\Delta$ and all edge resistances are between $1$ and $r$), we take all edge resistances equal to $1$ and assume all vertices in $T$ have degree $\Delta$, which gives us an effective resistance equal to $\frac{\Delta-1}{\Delta-2}$.  Likewise, for an upper bound on $\mathscr{R}\Big(v\underset{T_{v,v_j}}{\longleftrightarrow}\infty\Big)$, we take all edge resistances equal to $r$ and assume all vertices in $T$ have degree $\delta$, which gives us an effective resistance equal to $\frac{r(\delta-1)}{\delta-2}$.  Combining these bounds, we obtain the following upper and lower bounds on the expression from \eqref{transprle}:
	\begin{equation*}\frac{1}{2\Delta r}\leq\frac{1}{1+\frac{r(\delta-1)(\Delta-2)}{\delta-2}}\leq{\bf P}_v(X_1=v_i)\leq\frac{1}{1+\frac{(\delta-2)(\Delta-1)}{r(\Delta-2)}}\leq\frac{r}{r+\delta-2}
	\end{equation*}
	where the lower bound follows from plugging in the maximum possible value for $\mathscr{R}\Big(v\underset{T_{v,v_j}}{\longleftrightarrow}\infty\Big)$ for $j=i$ and the minimum possible for $j\neq i$, while setting $N=\Delta$, and the upper bound follows from plugging in the minimum possible value of $\mathscr{R}^{-1}\Big(v\underset{T_{v,v_j}}{\longleftrightarrow}\infty\Big)$ for $j=i$ and the maximum possible for $j\neq i$, while setting $N=\delta$.  Similarly, using \eqref{gtransprle}, along with the bound $\mathscr{R}\Big(v\underset{T}{\longleftrightarrow}\infty\Big)\geq\frac{\Delta-1}{\Delta(\Delta-2)}$ (in addition to the lower and upper bounds on $\mathscr{R}\Big(v\underset{T_{v,v_j}}{\longleftrightarrow}\infty\Big)$ derived above), we obtain a lower bound on $\Prob(X_{n+1}=v_{i'}|X_n=v,X_{n-1}=v_i)$ via the following string of inequalities:
	\begin{align*}\Prob(X_{n+1}=v_{i'}|X_n=v,X_{n-1}=v_i)&=\frac{1\Big/\mathscr{R}\Big(v\underset{T_{v,v_{i'}}}{\longleftrightarrow}\infty\Big)}{r_i\Big(r^{-1}_i+\underset{j\neq i}{\sum}r^{-1}_j\big(1-p(v_j,v)\big)\Big)\Big/\mathscr{R}\Big(v\underset{T_{v_i,v}}{\longleftrightarrow}\infty\Big)}\\
	&\geq\frac{\mathscr{R}\Big(v\underset{T_{v_i,v}}{\longleftrightarrow}\infty\Big)}{r\mathscr{R}\Big(v\underset{T_{v,v_{i'}}}{\longleftrightarrow}\infty\Big)\bigg(1+\mathscr{R}^{-1}\Big(v\underset{T}{\longleftrightarrow}\infty\Big)\bigg)}\\
	&\geq\frac{\frac{\Delta-1}{\Delta-2}}{r^2\frac{\delta-1}{\delta-2}\Big(1+\frac{\Delta(\Delta-2)}{\Delta-1}\Big)}\\
	&\geq\frac{1}{2\Delta r^2}.\end{align*}
	
	\end{proof}
	
	\bigskip
	\noindent
	\subsection{Proof of recurrence}  \label{sec:proof}
		
	Let $T$ be a tree in $\mathcal{T}_w$ (as defined in Section \ref{sec:truncated}) with maximum degree $\Delta$ and maximum edge resistance $r$.  Denoting the vertex set of $T$ as $\mathcal{V}$, we now define the probability space $$\Omega:=\bigg(\mathbb{N}\times\Big(\mathcal{V}^{(\mathbb{N}\setminus\{0\})\times\mathbb{N}}\Big)\times\Big([0,1]^{(\mathbb{N}\setminus\{0\})\times\mathbb{N}}\Big)\bigg)^{\mathcal{V}}$$ with probability measure $\TFM^{(\lambda_0)}$ under which all coordinates are independent, and where for each $v\in\mathcal{V}$ the three independent components are distributed in the following way: The first component is a random variable with distribution $\text{Poiss}(\lambda_o)$ representing the number of frogs originating at $v$ (unless $v$ is the root, in which case the first component is just the constant $1$). The second component represents the loop-erased random walks that will be performed by the particles originating at $v$ if they are activated (note that as a formal matter, these walks persist past the time at which they potentially terminate, and in fact each vertex is endowed with an infinite sequence of walks, rather than just those that correspond to the actual particles originating at that vertex). Finally, the third component is a collection of sequences of i.i.d. uniform $[0,1]$ random variables designed to break ties, in the sense that if the $j$th particle originating at $v$ jumps on its $i$th step onto a previously unvisited vertex at the same time as some other particle(s), then this particle terminates {\it unless} the value of the random variable associated with its $i$th step is greater than those of these other particles.
	
	Having now defined the probability space associated with the truncated frog model on $T$, we begin with a basic lemma: 
	
	\begin{lemma}\label{lem:basejg}
		Let $v \in T$ be a vertex of distance at least $2$ from the root and set $\lambda_0 = 2\Delta r \lambda$.  Then $$\TFM^{(\lambda_o)}(v \text{ is activated }\,|\,\Par{v} \text{ is activated} ) \geq 1 - e^{-\lambda}\,.$$
	\end{lemma}
	\begin{proof}
		Set $p$ to be the probability that a loop-erased random walk starting at $\Par{v}$ goes to $v$.  Then by Poisson thinning, the number of particles originating at $\Par{v}$ that go to $v$, conditioned on $\Par{v}$ being activated, is $\Poiss(p \lambda_0)$.  Alongside the bound from \eqref{lbpfsn}, this implies that conditioned on $\Par{v}$ being activated, the number of particles that hit $v$ stochastically dominates $\Poiss(\lambda)$, thus completing the proof.
	\end{proof}
	
\medskip
Now for each non-root vertex $v$ of $T$, let ${\sf TFM}^{(\lambda_o)}_v$ represent the measure associated with the truncated model on $T$, conditioned on the vertex $v$ eventually being landed on by the particle starting at the root (note that when we restrict our focus to the behavior of the truncated model inside $T(v)$, we can without loss of generality define ${\sf TFM}^{(\lambda_o)}_v$ by conditioning on the weaker assumption that $v$ is merely landed on by any particle).  Lemma \ref{lem:basejg} immediately implies that for every $u \in T(v)$, we have
	\begin{equation}\label{pagap}
	{\sf TFM}^{(\lambda_o)}_v(u\text{ is activated } | \Par{u} \text{ is activated}  )\geq 1-e^{-\lambda}.
	\end{equation} 
	Ultimately, the inequality in \eqref{pagap} will be used as the base case for an induction argument.  
	
	\medskip
	Much of the remainder of the proof of Theorem \ref{theorem:mrr} consists of establishing the following proposition:

	\begin{prop} \label{pr:iterate}
		Let $\mathscr{T} = \mathscr{T}(\delta,\Delta,r)$ denote the set of all weighted trees for which all vertices have degree between $\delta \geq 3$ and $\Delta$ and all resistances are between $1$ and $r$.  Then for $\lambda$ sufficiently large (and $\lambda_o=2\Delta r\lambda$) and each $i \geq 2$, the inequality 
		$${\sf TFM}^{(\lambda_o)}_v(u \text{ is activated } | \Par{u} \text{ is activated}  )\geq 1 - e^{-i \lambda / 2} $$
		for all trees $T \in \mathscr{T}$, vertices $v\in T\setminus\{{\bf 0}\}$, and $u \in T(v)$, in fact implies $${\sf TFM}^{(\lambda_o)}_v(u \text{ is activated } | \Par{u} \text{ is activated}  )\geq 1 - e^{-(i+1) \lambda / 2}$$
		for all $T \in \mathscr{T}$, $v\in T\setminus\{{\bf 0}\}$, and $u \in T(v)$.
	\end{prop}

	\begin{proof} The law of the loop erased random walk beginning at the root on a tree $T$ is called the \emph{harmonic measure}, which we will denote by ${\sf HARM}_T$.  Now we let $\alpha$ be some constant greater than $\lambda$ (with $\lambda$ to be chosen later and $\lambda_o$ again equal to $2\Delta r\lambda$), and we assume that
		\begin{equation} \label{eq:inductive}
		{\sf TFM}^{(\lambda_o)}_v(u \text{ is activated } | \Par{u} \text{ is activated}  ) \geq 1-e^{-\alpha}
		\end{equation} 
		for all $T \in \mathscr{T}$, $v\in T\setminus\{{\bf 0}\}$, and $u \in T(v)$.  Subject to this assumption, we now proceed to prove a series of three claims.
		
		\medskip
		\noindent
		{\bf Claim 1:} Let $\E^{(\lambda_o)}_v$ represent expectation with respect to ${\sf TFM}^{(\lambda_o)}_v$, and let $T_n(v)$ represent the vertices on level $n$ of the tree $T(v)$.  Then \begin{equation}\label{lbehs}\E^{(\lambda_o)}_v\bigg[\sum_{\substack{v'\in T_n(v) \\ v'\text{ is activated}}}{\sf HARM}_{T(v)}(v')\bigg]\geq(1-e^{-\alpha})^n\,.\end{equation}
		
		\bigskip
		\noindent
		{\sc Proof of Claim 1:} Let $v' \in T_n(v)$ and $v = u_0, u_1, \ldots, u_n = v'$ denote the path from $v$ to $v'$.  
		Iterating assumption \eqref{eq:inductive} implies 
		\begin{equation}\label{bcindnl}
		{\sf TFM}^{(\lambda_o)}_v(v' \text{ is activated})=\prod_{i=1}^n{\sf TFM}^{(\lambda_o)}_v(u_{i} \text{ is activated}\ |\  \Par{u_i} \text{ is activated} )\geq (1-e^{-\alpha})^n\,.
		\end{equation}
		Utilizing $\sum_{v' \in T_n(v)} {\sf HARM}_{T(v)}(v') = 1$ and taking expectation completes the proof of the claim. \hfill $\blacksquare$
		
		\bigskip
		Now if $n$ is chosen so it is $o(e^{\alpha})$, \eqref{lbehs} shows that a large proportion of vertices in $T_n(v)$---when weighted by the harmonic measure--- are typically activated.  For the proof of the proposition, we want to show that at least a constant proportion are activated with high probability.  This takes the form of the following claim:

		\medskip
		\noindent
		{\bf Claim 2:} Set $\beta := 1 - \frac{2r}{2r + \delta - 2} = \frac{\delta-2}{2r + \delta - 2}$; for $n\geq 0$, if we condition on the non-root vertex $v$ of $T$ being activated, then with probability at least $$1 - e^{-\alpha}\left(4 n^2 e^{-\alpha} + \frac{2n e^{-\alpha}}{\beta} + \frac{1}{\beta} \right)\,,$$
		there exists a vertex $v^* \in T_1(v)$ so that \begin{equation} \label{eq:half}
		\sum_{\substack{v' \in T_n(v^*) \\ v' \text{ is activated} }} {\sf HARM}_{T(v^*)}(v') \geq \frac{1}{2}\,.\end{equation}
		 
		\medskip
		\noindent
		{\sc Proof of Claim 2:} By the upper bound in \eqref{lbpfsn}, we see that ${\sf HARM}_{T(v)}(v')\leq 1-\beta$ for every vertex $v'\in T_1(v)$ (observe that we've replaced $r$ by $2r$ in this upper bound in order to account for the fact that the vertex $v$ may only have degree $\delta-1$ in $T(v)$).  Hence, it follows that if only one vertex in $T_1(v)$ is activated, then the sum in \eqref{lbehs}, for the case where $n=1$, is no greater than $1-\beta$.  Combining this with the lower bound on the expectation given in \eqref{lbehs} (again for $n=1$), we can then conclude that \begin{equation}\label{ubova}{\sf TFM}^{(\lambda_o)}_v\big(\#\{\text{activated vertices in }T_1(v)\}\leq 1\big)\leq{\sf TFM}^{(\lambda_o)}_v\bigg(\sum_{\substack{v'\in T_1(v) \\ v'\text{ is activated}}}{\sf HARM}_{T(v)}(v')\leq 1-\beta\bigg)\leq\frac{e^{-\alpha}}{\beta}.\end{equation}Next let $v_o$ represent the vertex in $T_1(v)$ that is hit by the particle which activated $v$, and let $A$ represent the event that at least two vertices in $T_1(v)$ are activated.  If we now take any positive integer $n$ for which $\log n<\alpha$, then using \eqref{lbehs} and \eqref{ubova}, we get
		\begin{align}
		{\sf TFM}^{(\lambda_o)}_v\bigg(\sum_{\substack{v'\in T_n(v_o) \\ v'\text{ is activated}}}{\sf HARM}_{T(v_o)}&(v')\geq\frac{1}{2}\bigg|A\bigg)\geq 1-{\sf TFM}^{(\lambda_o)}_v(A^c)\nonumber\\&-\sum_{v_i\in T_1(v)}{\sf TFM}^{(\lambda_o)}_v(v_o=v_i)\cdot{\sf TFM}^{(\lambda_o)}_{v_i}\bigg(\sum_{\substack{v'\in T_n(v_i) \\ v'\text{ is activated}}}{\sf HARM}_{T(v_i)}(v')<\frac{1}{2}\bigg)\nonumber\\ 
		&\geq 1-\frac{e^{-\alpha}}{\beta}-2n e^{-\alpha}\label{lbsgohga}
		\end{align}
		where the last term in the expression on the third line follows from \eqref{lbehs}, along with the fact that $(1-e^{-\alpha})^n\geq 1-n e^{-\alpha}$.  Letting $v''$ represent the left most vertex in $T_1(v)\setminus\{v_o\}$ that is hit by either a particle originating at $v$, or a particle originating in $T(v_o)$, and denoting the events
		$$\left\{\underset{\substack{v'\in T_n(v_o) \\ v'\text{ is activated}}}{\sum}{\sf HARM}_{T(v_o)}(v')\geq\frac{1}{2}\right\}\ \ \text{and  }\left\{\underset{\substack{v'\in T_n(v'') \\ v'\text{ is activated}}}{\sum}{\sf HARM}_{T(v'')}(v')\geq\frac{1}{2}\right\}$$
		as $B$ and $C$ respectively (note that $C\subseteq A$), we observe that
		\begin{align}\label{bbciccga}&{\sf TFM}^{(\lambda_o)}_v(B^c\cap C^c|A)={\sf TFM}^{(\lambda_o)}_v(B^c|A)\cdot{\sf TFM}^{(\lambda_o)}_v(C^c|A\cap B^c)\\ &={\sf TFM}^{(\lambda_o)}_v(B^c|A)\cdot\sum_{v_j\in T_1(v)}{\sf TFM}^{(\lambda_o)}_v(v''=v_j|A\cap B^c)\cdot{\sf TFM}^{(\lambda_o)}_v\Bigg(\underset{\substack{v'\in T_n(v'') \\ v'\text{ is activated}}}{\sum}{\sf HARM}_{T(v'')}(v')<\frac{1}{2}\Bigg|v''=v_j\Bigg).\nonumber
		\end{align}
		Using the bound in \eqref{lbsgohga} in conjunction with the equality in \eqref{bbciccga}, we can now conclude that $${\sf TFM}^{(\lambda_o)}_v(B^c\cap C^c|A)\leq 2n e^{-\alpha}\Big(\frac{e^{-\alpha}}{\beta}+2n e^{-\alpha}\Big)$$where the second term in the product bounds ${\sf TFM}^{(\lambda_o)}_v(B^c|A)$, and the first term serves as a bound, that follows from \eqref{lbehs}, on the sum in \eqref{bbciccga}.  Alongside the upper bound on ${\sf TFM}^{(\lambda_o)}_v(A^c)$ given in \eqref{ubova}, this then implies that\begin{equation}\label{lbbucia}{\sf TFM}^{(\lambda_o)}_v(B\cup C)\geq{\sf TFM}^{(\lambda_o)}_v(\{B\cup C\}\cap A)\geq 1-e^{-\alpha}\Big(4n^2 e^{-\alpha}+\frac{2n e^{-\alpha}}{\beta}+\frac{1}{\beta}\Big),\end{equation}thus completing the proof of the second claim. \hfill $\blacksquare$
		
		\bigskip
		\noindent
		Claim 2 is actually stronger than it may initially appear; in fact, \eqref{eq:half} implies
		\begin{equation}\label{lbshj0tn}
		\sum_{j=0}^n\sum_{\substack{v'\in T_j(v^*) \\ v'\ \text{is activated}}}{\sf HARM}_{T(v^*)}(v')\geq\frac{n+1}{2}
		\end{equation} since the value of the inner sum above is decreasing with respect to $j$.  
		
		\bigskip
		We now want to look at the number of particles from inside $T(v)$ that return to $\Par{v}$.
		
		\medskip	
		\noindent
		{\bf Claim 3:} Let $V_v$ represent the number of active particles originating inside $T(v)$ that eventually hit $\Par{v}$.  Then there exists $C'>0$ such that, conditioned on \eqref{lbshj0tn}, we have \begin{equation}\label{cnphrslb}V_v\succeq\Poiss\Big(C'\lambda_o\frac{n+1}{2}\Big)\,.\end{equation}
		
		\bigskip
		\noindent
		{\sc Proof of Claim 3:} For a vertex $v' \in T_j(v^*)$ for $j \leq n$, let $p_v(v')$ denote the probability that a loop-erased random walk beginning at $v'$ hits $\Par{v}$.  Then again by Poisson thinning, conditioned on $v'$ being activated, the number of particles originating at $v'$ that hit $\Par{v}$ is $\Poiss(p_v(v') \lambda_0)$.  Call this random variable $N_{v,v'}$; since no particles in $T(v^*)$ that reach $\Par{v}$ activate any other particles along the way, we in fact have that conditioned on some subset $S \subset T(v^*)$ being activated, the family of random variables $\{N_{v, v'}\}_{v' \in S}$ are mutually independent.  Further, $p_v(v')$ and ${\sf HARM}_{T(v^*)}(v')$ differ by only a bounded multiplicative constant (this follows from applying Lemma \ref{lem:crphmjg} if we formally designate $\Par{v}$ as the new root of $T$), i.e. there exists a constant $C > 0$ so that for all $v' \in T(v^*)$ we have $p_v(v') \geq C \cdot {\sf HARM}_{T(v^*)}(v')$.  Hence, the proof of the claim is complete. \hfill $\blacksquare$
		
		\bigskip
		In order to use \eqref{cnphrslb} to complete the proposition, we start by letting $u \in T(v)$ with $u \neq v$.  
		Recall by \eqref{pagap},   
		if $\Par{u}$ is activated, then the probability that $u$ is hit by either the particle that activated $\Par{u}$, or one of the particles originating at $\Par{u}$, is at least $1-e^{-\lambda}$.  If none of these particles hit $u$, then the particle that activated $\Par{u}$ must travel to one of the siblings $u'$ of $u$ since the particles never backtrack. We then find that
		\begin{equation}\label{lbpnlgejp1gnlgej}
		{\sf TFM}^{(\lambda_o)}_v(u \text{ activated}\,|\,\Par{u} \text{ activated})\geq 1-e^{-\lambda}+e^{-\lambda}{\sf TFM}^{(\lambda_o)}_v(u\ \text{hit by particle from }T(u')\,|\,u'\ \text{activated}).
		\end{equation}
		Since Claims 2 and 3 
		together imply that, if $u'$ is activated, then with probability at least $1-e^{-\alpha}\Big(4n^2 e^{-\alpha}+\frac{2n e^{-\alpha}}{\beta}+\frac{1}{\beta}\Big)$, the number of particles originating in $T(u')$ that hit $\Par{u}$ stochastically dominates $\text{Poiss}\Big(C'\lambda_o\frac{n+1}{2}\Big)$, and because \eqref{bfgtransprle} implies that a particle that travels from $u'$ to $\Par{u}$, then hits $u$ with probability at least $\frac{1}{2\Delta r^2}>0$, this means that there exists $C''>0$ such that, with probability at least $1-e^{-\alpha}\Big(4n^2 e^{-\alpha}+\frac{2n e^{-\alpha}}{\beta}+\frac{1}{\beta}\Big)$, the number of particles originating in $T(u')$ that hit $u$ (conditioned on $u'$ being activated) stochastically dominates $\text{Poiss}\Big(C''\lambda_o\frac{n+1}{2}\Big)$.  Hence, this implies that \eqref{lbpnlgejp1gnlgej} becomes
		\begin{equation}\label{slbpnlgejp1gnlgej}{\sf TFM}^{(\lambda_o)}_v(u \text{ activated}\,|\,\Par{u}\text{ activated})\geq 1-e^{-\lambda}\bigg(e^{-\alpha}\Big(4n^2 e^{-\alpha}+\frac{2n e^{-\alpha}}{\beta}+\frac{1}{\beta}\Big)+e^{-C''\lambda_o\frac{n+1}{2}}\bigg).
		\end{equation}
		Selecting $n=\lfloor{e^{\alpha/4}\rfloor}$, \eqref{slbpnlgejp1gnlgej} then gives us\begin{equation}\label{slbnlgnlp1}{\sf TFM}^{(\lambda_o)}_v(u \text{ activated}\,|\,\Par{u}\text{ activated})\geq 1-e^{-\lambda}\Big(4e^{-\frac{3\alpha}{2}}+\frac{2e^{-\frac{7\alpha}{4}}}{\beta}+\frac{e^{-\alpha}}{\beta}+e^{-\frac{C''\lambda_o}{2}e^{\alpha/4}}\Big)\geq 1-e^{-(\alpha+\frac{\lambda}{2})}\end{equation}
		where the second inequality holds for all $\alpha\geq\lambda$, for $\lambda$ sufficiently large.  Therefore, this completes the proof of the proposition.
	\end{proof}
	\medskip
	\begin{proof}[Proof of Theorem \ref{theorem:mrr}:]
		Combining Proposition \ref{pr:iterate} with \eqref{pagap}, we now see that, for $\lambda$ sufficiently large, $${\sf TFM}^{(\lambda_o)}_v(u \text{ activated}\,|\,\Par{u} \text{ activated})=1$$ for every $u \in T(v)$.  Hence, if we let $v$ represent the first vertex hit by the particle originating at the root, then the entire subtree rooted at $v$ is activated with probability $1$, or equivalently, the inequality in \eqref{lbshj0tn} holds for every $n$ with probability $1$.  Combining this with \eqref{cnphrslb}, we then see that the number of returns to the root almost surely dominates a Poisson random variable of arbitrarily large mean, thus establishing recurrence and completing the proof. 
	\end{proof}

	\section{Transience}
	\label{sec:transience}
	In this section we establish a transient regime for a wide class of frog models on non-amenable graphs.  This completes the second half of the proof of Theorem \ref{theorem:mrecth}, and in the process proves a lower bound on $\lambda_1(T)$.
	
	\noindent
	\subsection{Transience and the edge expansion constant} \label{sec:transience-isoperimetric} Central to our analysis in this section will be the notion of edge non-amenability.  We review some definitions first:
	
	\begin{defn}[Isoperimetric language from Chapter 6 of \cite{LP}] Let $G$ be a connected, locally finite graph.  
		\begin{itemize}
			\item For a finite subset $K \subset G$, define $|\partial K|$ to be the number of edges with precisely one vertex in $K$ and one in $K^c$.
			
			\item Define $|K|_D = \sum_{v \in K} \deg(v)$.
			
			\item Define the \emph{edge-expansion constant} to be $$\Phi_E(G) := \inf\left\{\frac{|\partial K| }{ |K|_D} : \emptyset \neq K \subset G \text{ is finite and connected}   \right\}\,.$$ 
			
			\item $G$ is \emph{edge amenable} if $\Phi_E(G) = 0$.  Otherwise, $G$ is \emph{edge non-amenable}.
		\end{itemize}
	\end{defn}
	
	\noindent
	Non-amenability will be enough to establish a transient regime of the frog model:
	
	\begin{theorem} \label{th:transience}
		Let $G$ be a connected, locally finite, and edge non-amenable graph. Fix a vertex $\rtt$ in $G$ and at each vertex $v\in G\setminus\{\rtt\}$, place $X_v \geq 0$ sleeping frogs where $\{X_v\}$ are jointly independent, and place a single awake frog at $\rtt$.  Then there exists $\lambda_0 > 0$ so that if $\E[X_v] \leq \lambda_0$ for all $v$, then only finitely many frogs visit the root almost surely.  Further, $\lambda_0 = \frac{\Phi_E(G)^2}{2 - \Phi_E(G)^2}$ is sufficient.
	\end{theorem}
	
	Our strategy for proving Theorem \ref{th:transience} will involve coupling the frog model with a branching random walk.  The isoperimetric property of edge non-amenable graphs comes into play due to the following relationship with simple random walk:  
	
	\begin{lemma}[Theorem 6.7 of \cite{LP}]\label{lem:isoperimetric}
		Let $G$ be a locally finite connected graph and fix some vertex $\rtt \in G$.  Let $\{Y_n\}$ denote a simple random walk on $G$ starting at $\rtt$, and set $\rho(G) = \limsup_{n \to \infty} \Prob[Y_{2n} = \rtt]^{1/2n}$.  Then $$ \Phi_E(G)^2 /2 \leq 1 - \rho(G) \leq \Phi_{E}(G)\,.$$
	\end{lemma}
	
	\medskip
	The quantity $\rho(G)$ is known as the \emph{spectral radius} of $G$, and is in fact independent of the choice of $\rtt$.  An important consequence of Lemma \ref{lem:isoperimetric} is that edge non-amenability is equivalent to $\rho(G) < 1$.  From here, a proof of Theorem \ref{th:transience} follows from comparison to a branching random walk.  Heuristically, the average number of frogs at time $2n$ will be at most $(1 + \lambda)^{2n}$ and the probability of being at the origin will be roughly $\rho^{2n}$.  If $\lambda$ is small enough, then the average number of particles at the root at time $2n$ will decay exponentially, and thus be summable.  
	
	\bigskip
	\noindent
	\emph{Proof of Theorem \ref{th:transience}:}  We will compare the frog model to a branching random walk.  Begin with a single particle at $\rtt$, which will perform a simple random walk.  When a particle lands on a vertex $v$, independently sample a copy of $X_v$, and add $X_v$ many particles at $v$, which in turn perform simple random walks.  The number of particles that return to $\rtt$ in this model stochastically dominates the number of frogs that return to $\rtt$ in the frog model, and thus it is sufficient to show that only finitely many particles visit the root in our new branching model.  This is a branching Markov chain, where the mean at each vertex is $1 + \E[X_v] \leq 1 + \lambda_0$.  By Theorem 3.2 of \cite{GM}, there are only finitely many visits to the root, provided $1 + \lambda_0 \leq \frac{1}{\rho(G)}$.   Lower bounding $$\frac{1}{\rho(G)} = \frac{1}{1 - (1 - \rho(G))} \geq \frac{1}{1-\Phi(G)^2 /2}$$ by Lemma \ref{lem:isoperimetric} completes the proof. \qed
	
	\bigskip
	For our main application of Theorem \ref{th:transience}, we will apply it to the class of trees $\mathcal{T}^{**}_f$: defined as $\underset{L\geq 1}{\cup}\mathcal{T}^*_L$, where $\mathcal{T}^*_L$ represents the set of all trees satisfying the two conditions in Lemma \ref{lem:non-amenable} each with $M = L$, and with no degree restriction.  To do this, we start by showing that the edge-expansion constant can be uniformly bounded below for trees in $\mathcal{T}^*_L$.
	
	\begin{lemma} \label{lem:tree-iso}
		For each $T \in \mathcal{T}^*_L$, $\Phi_E(T) \geq \frac{1}{9L^2}$.
	\end{lemma}
	
	\begin{proof}
		Let $K$ be a finite connected subset of $T$.  Suppose that $K$ contains elements $v_1,\ldots, v_k$ of $T_r$ so that each $v_j$ has degree at least $3$ in $T_r$.  We will prove the lemma by inducting on $k$, first starting with the case of $k = 0$.  For each $v \in T$, the connected component of $v$ in $(T \setminus T_r) \cup \{v\}$ is a tree with at most $L$ vertices, i.e. at most $L-1$ edges.  Therefore, the sum of degrees in this tree is at most $2(L-1)$.  In the case where $K \cap T_r = \emptyset$, we have $|K|_D \leq 2(L-1)$; otherwise, $$|K|_D \leq \sum_{v \in K \cap T_r} (2 + 2(L-1)) \leq 2 L^2\,. $$Noting $|\partial K| \geq 1$ shows the lemma in this case.  
		
		We will also need to address the $k = 1$ case; we may decompose $K$ into a single vertex $v \in T_r$ with degree at least $3$ in $T_r$ together with paths in $T_r$ coming off of $v$, with all of these vertices possibly decorated with finite trees.  Note that $|\partial K| \geq \deg_{T_r}(v)$.  For an upper bound on $|K|_D$ note first that the degree of $v$ plus the sums of the degrees of all the other vertices in the connected component of $v$ in $(T \setminus T_r ) \cup \{ v\}$ is at most $2L + \deg_{T_r}(v)$.  Together with the $\deg_{T_r}(v)$ many paths in $K \cap T_r$, we have $$|K|_D \leq 2L + \deg_{T_r}(v) + \deg_{T_r}(v)\left(2 L^2  \right) \leq (3 L^2) \deg_{T_r}(v)$$where the first inequality follows from the $k = 0$ case.  
		
		Next, assume that the lemma holds if $K$ contains at most $k \geq 1$ elements in $T_r$ with degree at least $3$ (in $T_r$); suppose $K$ now contains $k+1$ such elements.  Label them $v_1,\ldots,v_{k+1}$ in such a way that each $v_i$ is connected to $v_j$ in $K \setminus \{v_{k+1}\}$ for $i,j \leq k$.  Set $K_1$ to be the connected component in $K \setminus v_{k+1}$ that contains $v_1,\ldots,v_k$.   By the inductive hypothesis, $|\partial K_1|/ |K_1|_D \geq \frac{1}{9L^2}$.  When adding the remaining portion of $K$, we decrease the size of the boundary by $1$ since we've added $v_{k+1}$, but also add at least $\deg_{T_r}(v_{k+1}) -1$ elements to the boundary, so $$|\partial K| \geq |\partial K_{1}| + \deg_{T_r}(v_{k+1}) - 2\,.$$
		Similarly, the total addition to the degree can be bounded above via \begin{align*}|K|_D &\leq |K_1|_D + 2L + \deg_{T_r}(v_{k+1}) + \deg_{T_r}(v_{k+1})(2 L^2) \\&\leq |K_1|_D + (3 L^2) \deg_{T_r}(v_{k+1})  \,.\end{align*} Utilizing the inequality $x - 2 \geq x/3$ for $x \geq 3$ completes the inductive step, and thus the lemma.	
	\end{proof}
	
\medskip
A transient regime for the frog model on trees in $\mathcal{T}^*_L$ follows immediately from Theorem \ref{th:transience} and Lemma \ref{lem:tree-iso}.
	
	\medskip
	\begin{cor}\label{cor:jtlm}
		For each $T \in \mathcal{T}_L^*$, the frog model on $T$ is transient provided the distribution of frogs has mean at most $\frac{1}{162 L^4}$.
	\end{cor}
	
	\section{Examples and open questions}
	We close by providing a couple of examples to show that the assumptions in Theorem \ref{theorem:mrecth} cannot be removed, and by posing some open questions.
	
	\subsection{Examples}
		
	In Theorem \ref{theorem:mrecth}, the assumption of non-amenability cannot be removed: the graph $\mathbb{Z}$ is a tree of bounded degree and since simple random walk is recurrent on $\mathbb{Z}$, the frog model will be as well for any non-trivial i.i.d.~distribution of frogs.  On the other hand, if we drop the assumption that $T$ is of bounded degree, then there need not be a recurrent regime.  We prove this Lemma in \cite{mr}, but import the example here:
	
	\begin{lemma}[\cite{mr}] \label{lem:rtnrr}
		Let $T$ be the rooted tree where all vertices at depth $n$ have $n+2$ children.  Then the frog model with i.i.d.~$\Poiss(\lambda)$ frogs per vertex is transient on $T$ for all $\lambda$.
	\end{lemma}

	Further, it need not be the case that $\lambda_1$ and $\lambda_2$ are equal in Theorem \ref{theorem:mrecth}.  Indeed, it may be the case that there is a non-trivial region between the two on which infinitely many frogs visit the root with probability strictly between $0$ and $1$.  Again, we import the example from \cite{mr}
	
	\begin{lemma}[\cite{mr}]\label{lem:nontrivial-intermediate}
		Let $T$ denote the rooted tree formed by joining each of the roots of the $2$-ary tree and $d$-ary tree to a single root by a pair of distinct edges.  For $d$ sufficiently large, $\lambda_1(T) < \lambda_2(T)$.
	\end{lemma}

	\subsection{Further questions}
	There are many natural questions that emerge from Theorem \ref{theorem:mrecth}; we highlight a couple that seem out of reach using the methods of this work:
	
	\begin{question}
		In the case of trees, can the assumption of non-amenability be weakened?  More concretely, suppose $T$ is a tree on which simple random walk a.s.~has speed bounded uniformly away from $0$.  Must there exist a transient regime for the frog model on all such trees?
	\end{question}

	The assumption that $T$ is a tree is central to our analysis here.  It is therefore natural to ask whether these results can be extended to other families of graphs.
	
	\begin{question}
		On what classes of graphs is non-amenability sufficient to establish a recurrent regime?  
	\end{question}

	\noindent

	\bigskip
	\appendix
	
	\section{Harmonic measure and return probability} \label{app:hitting} 
	
	\bigskip
	The following lemma closely resembles Lemma A.1 from our paper \cite{mr}.  While large segments of the two proofs are in fact identical, the results are nevertheless distinct, with the result from \cite{mr} applying to trees of unbounded degree, and this result applying to trees of bounded degree with weighted edges.  In the statement of the lemma, $\mathcal{T}_w$ represents the set of weighted trees defined in Section \ref{sec:truncated}, and for any weighted tree $T$ and vertex $u\in T$, $p_0(u)$ is defined as the probability that loop erased random walk, starting at $u$, ever reaches the root of $T$.
	
	\medskip
	\begin{lemma}\label{lem:crphmjg}
		If a tree $T\in\mathcal{T}_w$ has minimum degree $\delta\geq 3$, maximum degree $\Delta<\infty$, and edge resistances between $1$ and  $r<\infty$, then there exists a value $C>0$ (depending only on $\Delta$ and $r$) such that for every $v$ on level $2$ of $T$ and every $u\in T(v)$, we have $$p_0(u)\geq C\cdot{\sf HARM}_{T(v)}(u).$$

	\end{lemma}
	
	\begin{proof}
		We begin by defining the following quantities: First, let $\tilde{p}(v,u)$ represent the probability that random walk on $T(v)$ beginning at $v$ (and performed in accordance with the assigned edge weights) ever hits $u$.  In addition, we define $\tilde{p}(u,\infty)$ to be the probability that random walk on $T(v)$ beginning at $u$ eventually escapes through one of the children of $u$.  Turning to random walk on $T$, we let $p(u,v)$ be defined as at the beginning of Section \ref{sec:looperased}, and we define $p(v,-\infty)$ to be the probability that random walk beginning at $v$ eventually escapes through one of the children of the root other than the parent of $v$.  Now noting that ${\sf HARM}_{T(v)}(u)=\tilde{p}(v,u)\cdot\tilde{p}(u,\infty)$ and $p_0(u)=p(u,v)\cdot p(v,-\infty)$, we see that we can prove the lemma by showing that $\frac{p(u,v)}{\tilde{p}(v,u)}\cdot p(v,-\infty)$ is bounded away from $0$.
		
		Looking first at $\frac{p(u,v)}{\tilde{p}(v,u)}$, we define $p^*(u,v)$ and $\tilde{p}^*(v,u)$ to be the probabilities that random walk on $T(v)$ beginning at $u$ ($v$ respectively) reaches $v$ ($u$ respectively) without first returning to its starting position.  Noting that
		\begin{equation}\label{inbnd1t1}\frac{p(u,v)}{\tilde{p}(v,u)}\geq\frac{p^*(u,v)}{\tilde{p}^*(v,u)}\cdot\frac{\tilde{p}^*(v,u)}{\tilde{p}(v,u)},
		\end{equation}
		and observing that
		\begin{equation}\label{inbnd1t2}
		\frac{p^*(u,v)}{\tilde{p}^*(v,u)}=\frac{\sum c(\tilde{e}_i)}{\sum c(e_j)}\end{equation} 
		(where the $\tilde{e}_i$'s represent the edges that touch $v$ in $T(v)$ and the $e_j$'s represent the edges that touch $u$ in $T$), we see that since both $T(v)$ and $T$ have bounded degrees and edge weights between $\frac{1}{r}$ and $1$, it follows that, in order to show that $\frac{p(u,v)}{\tilde{p}(v,u)}$ is bounded away from $0$, it will suffice to prove this for $\frac{\tilde{p}^*(v,u)}{\tilde{p}(v,u)}$.  Now we let $p$ represent the probability that random walk on $T(v)$ beginning at $v$ ever returns to $v$, and let $p'$ represent the probability that random walk on $T(v)$ beginning at $v$ returns to $v$ without first hitting $u$.  Observing that $\tilde{p}(v,u)=\frac{\tilde{p}^*(v,u)}{1-p'}\leq\frac{\tilde{p}^*(v,u)}{1-p}$, we see that 
		\begin{equation}\label{inbnd1t3}
		\frac{\tilde{p}^*(v,u)}{\tilde{p}(v,u)}\geq 1-p=\frac{\mathscr{R}^{-1}\Big(v\underset{T(v)}{\longleftrightarrow}\infty\Big)}{\sum c(\tilde{e}_i)}.
		\end{equation}
		Since $T(v)$ has bounded degree, all vertices have at least two children, and all edge weights are less than or equal to $1$, this then implies that both $\mathscr{R}\Big(v\underset{T(v)}{\longleftrightarrow}\infty\Big)$ and $\sum c(\tilde{e}_i)$ are bounded above, thus implying that $\frac{\tilde{p}^*(v,u)}{\tilde{p}(v,u)}$ (and therefore $\frac{p(u,v)}{\tilde{p}(v,u)}$) is bounded away from $0$.  Finally, to show that $p(v,-\infty)$ is bounded away from $0$, and thus complete the proof of the lemma, we observe that 
		$$p(v,-\infty)=\Prob_v(X_1=\Par{v})\cdot\Prob(X_2={\bf 0}|X_1=\Par{v}, X_0=v)\geq\frac{1}{4\Delta^2r^3}$$
		(where the second inequality follows from combining the lower bounds in \eqref{lbpfsn} and \eqref{bfgtransprle}).

	\end{proof}

	\bibliographystyle{abbrv}
	\bibliography{frogGW}

\end{document}